\documentclass[oneside,english,reqno]{amsart}
\usepackage[T1]{fontenc}
\usepackage[latin9]{inputenc}
\usepackage{amsthm}
\usepackage{amssymb}
\usepackage{amsmath}
\usepackage{color}
\makeatletter
\numberwithin{equation}{section}
\numberwithin{figure}{section}
\theoremstyle{plain}
\newtheorem{thm}{Theorem}[section]
\newtheorem{prop}[thm]{Proposition}

\newtheorem{conjecture}[thm]{Conjecture}
\newtheorem{lem}[thm]{Lemma}

\newtheorem{rem}[thm]{Remark}

  \newcounter{casectr}
  
\theoremstyle{definition}
\newtheorem{assumption}{Assumption}
\theoremstyle{remark}

\newcommand{\RRR}{\mathbf{R}}

\newcommand{\hdot}{\dot{H}^{1}(\mathbf{R}^{3})}
\newcommand{\hhh}{\mathbf{H}^{3}}

\newcommand{\hlt}{L^{2}({\hhh})}
\newcommand{\hls}{L^{6}({\hhh})}
\newcommand{\limk}{\lim_{k\rightarrow \infty}}

\newcommand{\signot}{\sigma_{0}}
\newcommand{\hnorm}{H^{1}(\hhh)}
\newcommand{\liminfk}{\liminf_{k\rightarrow\infty}}
\newcommand{\limsupk}{\limsup_{k\rightarrow\infty}}
\newcommand{\hyperenergy}{E_{\hhh}}
\newcommand{\rrr}{\mathbf{R}^{3}}
\newcommand{\rls}{L^{6}(\mathbf{R}^{3})}
\newcommand{\rlt}{L^{2}(\mathbf{R}^{3})}

\newcommand{\ek}{E_{\hhh}({u_{k}})}

\makeatother

\usepackage{fancyhdr}
\usepackage{babel}
\providecommand{\casename}{Case}
\begin{document}
\title{ On a variational problem related to NLS on Hyperbolic space}
\begin{center}
\author{Chenjie Fan $^{\ddagger} $ $^\dagger$}, \author{ Peter Kleinhenz$^\sharp$}
\end{center}
\thanks{\quad \\$^{\ddagger} $Department of Mathematics,  Massachusetts Institute of Technology,  77 Massachusetts Ave,  Cambridge,  MA 02139-4307 USA. email:
cjfan@math.mit.edu.\\
$^\dagger$ The author is partially supported by NSF Grant DMS 1069225, DMS 1362509 and DMS 1462401.\\
$^\sharp$ Northwestern University, Department of Mathematics, 2033 Sheridan Rd, Evanston Il, 60201 USA. email:pbk@math.northwestern.edu}

\maketitle
\begin{abstract}
Motivated by NLS, We study a variational problem on hyperbolic space. In particular, we  compute its minimum value and we show the minimizer does not exist.
\end{abstract}

\section{Introduction}
\subsection{Statement of problem and results}
In this short note, we consider the following minimizing problem on 3D hyperbolic space:
\begin{equation}
\begin{cases}
J(f):=\frac{E_{\hhh}(f)}{\|f\|_{\hlt}^{2}}, f\in H^{1}_{rad}(\hhh),\\
E_{\hhh}(f)<E_{\RRR^{3}}(Q),\\
\|\nabla_{\hhh} f\|_{L^{2}(\hhh)}<\|\nabla Q\|_{L^{2}(\RRR^{3})}.
\end{cases}
\end{equation}

Here
$$E_{\hhh}(f):=\frac{1}{2}\|\nabla_{\hhh} f\|_{L^{2}(\hhh)}^{2}-\frac{1}{6}\|f\|^{6}_{L^{6}(\hhh)},\quad  E_{\RRR^{3}}(Q):=\frac{1}{2}\|\nabla Q\|_{L^{2}(\RRR^{3})}^{2}-\frac{1}{6}\|Q\|_{L^{6}(\RRR^{3})}^{6}, $$
$$  H^1_{rad}(\hhh) := \{ f \in H^1(\hhh); f \text{ is radial} \}, $$
and $Q$ is the ground state on Euclidean space, i.e. the unique positive $\dot{H}^{1}$ radial solution to\footnote{When one talks about positive solutions, there is no difference between $|W|^{4}W$ and $W^{5}$.} 
\begin{equation}
\label{eucgroundstate}
\Delta W+W^{5}=0.
\end{equation}
We show 
\begin{thm}\label{thmmain}
Let $\Omega  \subset  H^{1}_{rad}(\hhh)$ be defined as
$$\Omega:=\{f|f\in H^{1}_{rad}(\hhh), f\neq 0,  E_{\hhh}(f)<E_{\RRR^{3}}(Q), \|\nabla_{\hhh} f\|_{L^2(\hhh)}<\|\nabla Q\|_{L^{2}(\RRR^{3})}\}.$$
Then we have
\begin{equation}
\inf_{f\in \Omega}J(f)=\frac{1}{2}
\end{equation}
\end{thm}
\subsection{Motivation}

We are interested in $J$ because of its connection to the 
focusing energy critical NLS on hyperbolic space:
\begin{equation}\label{hnls}
\begin{cases}
iu_{t}+\Delta_{\hhh} u=-|u|^{4}u,\\
u_{0}\in H^{1}(\hhh).
\end{cases}
\end{equation}
Motivated by the seminal result of Kenig and Merle, \cite{kenig2006global}, one might want to show the following:
\begin{conjecture}[formal conjecture]\label{conj1}
Let $u$ be the solution to \eqref{hnls}, such that
$$E_{\hhh}(u_{0})<E_{\hhh}(Q_{*}), \quad \|\nabla_{\hhh}u_{0}\|_{L^{2}(\hhh)}<\|\nabla _{\hhh}Q_{*}\|_{L^{2}(\hhh)},$$ where $Q_{*}$ is the unique\footnote{In hyperbolic space, $H^{1}(\hhh)$ and $\dot{H}^{1}(\hhh)$ are the same space.} $H^{1}(\hhh)$ positive solution to the ground state equation on hyperbolic space 
$$\Delta_{\hhh}  Q_{*} + Q_{*}^{5}=0.$$
Then $u$ is global and scatters in the sense that
\begin{equation}
\|u(t,x)\|_{L^{10}_{\RRR\times \hhh}}\lesssim 1.
\end{equation}
\end{conjecture}
This conjecture, seemingly natural, is not even correctly stated. {As} shown in \cite{mancini2008semilinear}, there is no nontrivial $H^{1}(H^{3})$ positive solution to 
\begin{equation}
\begin{cases}
\Delta_{\hhh} f+\lambda f+f^{5}=0,\\
f\in H^{1}(\hhh), \lambda\in \RRR.
\end{cases}
\end{equation}

Let us reconsider the NLS on hyperbolic space. Roughly speaking, the road map suggested by \cite{kenig2006global} has three ingredients:
\begin{enumerate}
\item  Variational inequality/ energy trapping argument.
\item Concentration compactness/profile decomposition.
\item Rigidity argument.
\end{enumerate}
In the case of NLS on 3D hyperbolic space, the variational inequality, which is related to the  ground state in Euclidean case, seems to collapse at first glance since there is no ground state here. Concentration compactness is fine, see \cite{ionescu2012global}. The rigidity argument, which is related to virial identity, is much more complicated than in the Euclidean space.

However, one can make the observation that the variational inequality argument in \cite{kenig2006global} is only related to the fact that some of the ground state's norms can be expressed in terms of the best constant for the Sobolev inequality. Thus one possible way to modify Conjecture \ref{conj1} is the following:
\begin{conjecture}\label{mainconj}
Let $u$ be the solution to \eqref{hnls}
, such that
\begin{equation}\label{maincondition}
E_{\hhh}(u_{0})<E_{\rrr}(Q), \quad \|\nabla_{\hhh} u_{0}\|_{L^{2}(\hhh)}<\|\nabla Q \|_{L^{2}(\rrr)},
\end{equation}
where $Q $ is the unique positive $\dot{H}^{1}(\RRR^{3})$ radial solution to the ground state equation on Euclidean space
$$\Delta Q+Q^{5}=0.$$
Then $u$ is global and scatters in the sense that
\begin{equation}
\|u(t,x)\|_{L^{10}_{\RRR\times \hhh}}\lesssim 1.
\end{equation}
\end{conjecture}
One potential obstacle to prove this result is Banica's result \cite{banica2007nonlinear}, which implies the following:
\begin{thm}
Assume $u_{0}\in H^{1}_{rad}(\hhh)$, such that
$$E_{\hhh}(u_{0})<\frac{1}{2}\|u_{0}\|_{L^{2}(\hhh)}^{2}, \|ru_{0}\|_{L^{2}(\hhh)}<\infty.$$
Let $u$ be the solution to \eqref{hnls} with initial data $u_{0}$, then $u$ must blow up in finite time. 
\end{thm}
So {before attempting to prove} Conjecture  \ref{mainconj}, one would like to know if the condition \eqref{maincondition} excludes the case {$E_{\hhh}(u_{0})<\frac{1}{2}\|u_{0}\|_{L^{2}(\hhh)}^{2}$.}

This is our motivation for studying the minimizing problem of the functional $J$, and our main theorem gives a positive answer.
\begin{rem}
If one is familiar with the results of \cite{kenig2006global} and the profile decomposition in the hyperbolic space \cite{ionescu2012global}, then one should be able to follow the argument in \cite{ionescu2012global} to show all data described in Conjecture \ref{mainconj} generates global solutions, which indeed implies our Theorem. Nevertheless, we will give a self contained proof, which is independent of the NLS. Note, the truly subtle problem, is to show the data described in Conjecture \ref{mainconj} will scatter. To show this, one will need some new computation related to virial identity. We do not discuss this matter here. One may see \cite{banica2014global} and the reference in that article for this direction.
\end{rem}
\subsection{Notations}
We view the 3D hyperbolic space in polar coordinates $(r,w)$ with Lorenzian metric on $\RRR^4$
\begin{equation}
ds^2 = dr^2 + \sinh^2 r d \omega^2,
\end{equation}
with the volume element 
\begin{equation}
\int_{\hhh} u(x) d \Omega = \int_0^{\infty} \int_{S^2} u(r,\omega) \sinh^2 r \, d \omega d r.
\end{equation}
We use the notation $u=u(x)=u(|x|)=u(r)$ for any radial function $u$. Note with the usual polar coordinates in $\RRR^{3}$, any function in hyperbolic space  can also be naturally regarded as a function in Euclidean space. We use $\nabla_{\hhh}$ to denote the co-variant derivative on hyperbolic space and we use $\nabla$ to denote the usual gradient on Euclidean space.
We write  $A\lesssim B$ when 
 $A\leq CB$, for some universal constant $C$,  we write $A\gtrsim B$ if $B\lesssim A$. We write $A\sim B$ if $A\lesssim B$ and $B\lesssim A$. 
\subsection{Preliminary}

\subsubsection{Sobolev embedding}
We quote Sobolev embedding results for $\rrr$ and $\hhh$. 
First by Aubin \cite{aubin1976espaces}
\begin{lem} 
\label{sobeuc}
Let $u \in \dot{H}^1(\rrr)$ then 
\begin{equation}
\label{sobeuceq}
||u||_{L^6(\rrr)} \leq C_3 ||\nabla u||_{L^2(\rrr)}.
\end{equation}
Moreover, if $||u||_{L^6(\rrr)} = C_3 ||\nabla u||_{L^2(\rrr)}$ and $u \neq 0$ then {$u$ is a solution to \eqref{eucgroundstate}}
\end{lem} 
Second by Hebey \cite{hebey}
\begin{lem}
\label{sobhyp}
Let $u \in H^1(\hhh)$ then 
\begin{equation}
\label{sobhypeq}
||u||_{L^6(\hhh)} \leq C_3 ||\nabla_{\hhh} u||_{L^2(\hhh)},
\end{equation}
where $C_3$ is the same constant as in Lemma \ref{sobeuc}. 
\end{lem}
We emphasize that the constants $C_3$ in \eqref{sobhypeq} and \eqref{sobeuceq} are the same. 

\subsubsection{Variational inequality and energy trapping}

We have the following lemma as in \cite{kenig2006global}
\begin{lem}\label{energytrapping}
Assume 
\begin{equation*}
|| \nabla_{\hhh} u ||_{L^2(\hhh)}^2 < ||\nabla Q||_{L^2(\rrr)}^2 \quad \text{and} \quad E_{\hhh}(u) <(1-\delta_0)E_{\rrr}(Q),
\end{equation*}
where $\delta_0>0$.
Then there exists $\bar{\delta} = \bar{\delta}(\delta_0)>0$ such that 
\begin{equation}
\label{graddelta}
{ ||\nabla_{\hhh} u||_{L^2(\hhh)}^2 < (1- \bar{\delta}) || \nabla Q||_{L^2(\rrr)}^2},
\end{equation}
and
\begin{equation}
\label{energyminus}
{||\nabla_{\hhh} u |_{L^2(\hhh)}^2 - ||u||_{L^6(\hhh)}^6 \geq \bar{\delta} ||\nabla_{\hhh} u||_{L^2(\hhh)}^2 ,}
\end{equation}
and in particular
\begin{equation}
\label{posenergy}
{E_{\hhh}(u) \gtrsim ||\nabla_{\hhh} u||_{L^2(\hhh)}^{2}.}
\end{equation}
\end{lem}
\begin{proof}
From the proof of the Lemma 3.4 in \cite{kenig2006global} we have that 
\begin{equation}
|| \nabla Q ||_{L^2( \rrr)}^2 = \frac{1}{C_3^3} \quad \text{and} \quad E(Q) = \left(\frac{1}{2}-\frac{1}{6} \right)|| \nabla Q ||_{L^2( \rrr)}^2 = \frac{1}{3 C_3^3},
\end{equation}
where $C_3$ is the best constant for the Sobolev inequality for $\rrr$. 
The remainder of the proof is almost identical to that of Lemma 3.4 of \cite{kenig2006global}, we quickly review it.
We consider the function $$f_1(y) = \frac{1}{2} y - \frac{C_3^6}{6} y^3, $$ which if we plug in $||\nabla_{\hhh} u||_{L^2(\hhh)}$ by  \eqref{sobhypeq} and \eqref{sobeuceq} {is less than or equal to $E_{\hhh}(u)$. This along with the fact that $f$ is strictly increasing for $0<y<1/C_3^3$ and that $E_{\hhh}(u) < (1-\delta_0) E_{\rrr}(Q)$} gives us \eqref{graddelta}.

We also consider $$g_1(y) = y - C_3^6 y^3, $$ which, if we again plug in $||\nabla_{\hhh} u||_{L^2}$, by  \eqref{sobhypeq} and \eqref{sobeuceq} is less than or equal to $\int_{\hhh} |\nabla_{\hhh} u|^2 - |u|^6$. We note that $g'(y)$ is bounded below by $-2$ for $0<y<1/C_3^3$ which along with \eqref{graddelta} {gives us a lower bound on $g(y)$ on the same interval  in terms of $||\nabla_{\hhh} u||_{L^2(\hhh)}^2$} which shows \eqref{energyminus}. 

{Once we have \eqref{energyminus}, then \eqref{posenergy} follows directly as 
$$ E_{\hhh}(u) = \left( \frac{1}{2} - \frac{1}{6} \right) ||\nabla_{\hhh} u||_{L^2(\hhh)}^2 + \frac{1}{6} \left( ||\nabla_{\hhh} u||_{L^2(\hhh)}^{2} - ||u||_{L^6(\hhh)}^{6}\right) \geq C_{\bar{\delta}} ||\nabla_{\hhh} u||_{L^2(\hhh)}^{2}. $$}

\end{proof}

\subsubsection{Profile decomposition on Euclidean space}

We recall the profile decomposition on Euclidean space:
\begin{lem}\label{concentrationcompactness}
{Let $u_{k}$ be a sequence of bounded radial $\hdot$ functions}, then (up to replacement by a subsequence), one can find
$\{v_{j};\{\lambda_{j,k}\}_{k}\}_{j}$ where $v_{j}$ is a radial $\hdot$ function and $\lambda_{j,k}\in \RRR_{+}$, such that for any $J\geq 1$,
\begin{equation}\label{profile}
u_{k}(x)=\sum_{j=1}^{J}\frac{1}{\lambda_{j,k}^{1/2}}v_{j} \left( \frac{x}{\lambda_{j,k}} \right)+R_{J,k},
\end{equation}
and the following holds:
\begin{itemize}
\item $j\neq j'$ implies $\limk |\ln (\lambda_{j,k}/\lambda_{j,k'})|=\infty$.
\item For any fixed $J$, $\|u_{k}\|_{\hdot}^2=\sum_{j=1}^{J}\|v_{j}\|_{\hdot}^{2}+\|R_{J,k}\|_{\hdot}^{2}+o_{k}(1)$.
\item $\lim_{J\rightarrow\infty}\limsup_{k \rightarrow \infty}\|R_{J,k}\|_{L^{6}(\RRR^{3})}=0$
\end{itemize}
We say $u_{k}$ admits a profile decomposition with $\{v_{j};\{\lambda_{j,k}\}_{k}\}_{j}$, and we call $v_{j}$ a profile.
\end{lem}
See \cite{gerard1998description}, \cite{keraani2001defect} for a proof.
\subsection{A nonexistence result for an elliptic PDE}
\begin{lem}\label{nonextencepre}
{There is no positive  solution to
\begin{equation}
\Delta_{\hhh} u+u^{5}+\lambda u=0,\quad \lambda\in \RRR,
\end{equation}
in $H^{1}(\hhh)$.}
\end{lem}
See \cite{mancini2008semilinear} for a proof.
We point out, by standard techniques from elliptic equations, one can upgrade the above lemma to
\begin{lem}\label{nonextence}
{There is no non-negative solution to
\begin{equation}
\Delta u+u^{5}+\lambda u=0, \quad \lambda\in \RRR,
\end{equation} 
in $H^{1}(\hhh) \backslash \{0\}$.}
\end{lem}
We will prove Lemma \ref{nonextence} in Appendix \ref{aaaaa}.
\section{Proof of main result}
\subsection{No minimizer can exist}
{We first show no element in $\overline{\Omega} \backslash \{0\}$ can minimize $J$.}
\subsubsection{No minimizer can exist in the interior of $\Omega$}
We show 
\begin{lem}\label{noninter}
There does not exist $u_{0}\in \Omega$ such that $J(u_{0})=\inf_{u\in \Omega }J(u)$.
\end{lem}
\begin{proof}
We first point out that we can assume $u_{0}$ is non-negative, since $J(|u_{0}|) \leq J(u_{0})$. Then,  we compute the Euler-Lagrange equation for the minimizer and we get for some $\lambda_{0}\in \RRR$
 \begin{equation}
 \Delta u_{0}+u_{0}^{5}+\lambda_{0}u_{0}=0.
 \end{equation}
 which contradicts Lemma \ref{nonextence}.
\end{proof}
\subsubsection{No minimizer can exist on  the boundary of $\Omega$ }
We show
\begin{lem}\label{nobound}
There does not exist $u_{0}\in \partial \Omega$ with $u_0 \neq 0$ such that $J(u)=\inf_{u\in \Omega}{J(u)}$
\end{lem}
\begin{proof}
Again we assume $u_{0}$ non-negative. There are three cases:
\begin{enumerate}
\item $\|\nabla_{\hhh} u_{0}\|_{L^{2}(\hhh)}=\|\nabla Q\|_{L^{2}(\RRR^{3})}, E_{\hhh}(u_{0})<E_{\RRR^{3}}(Q).$
\item $\|\nabla_{\hhh} u_{0}\|_{L^{2}(\hhh)}<\|\nabla Q\|_{L^{2}}(\RRR^{3}), E_{\hhh}(u_{0})=E_{\RRR^{3}}(Q).$
\item $\|\nabla_{\hhh} u_{0}\|_{L^{2}(\hhh)}=\|\nabla Q\|_{L^{2}}(\RRR^{3}), E_{\hhh}(u_{0})=E_{\RRR^{3}}(Q).$
\end{enumerate}
If case 1 holds, then by Lagrangian multiplier one derive for some $\lambda_{0}\in \RRR$
\begin{equation}\label{m1}
(1+M_{0} \lambda_{0})\Delta _{\hhh}u_{0}+u_{0}^{5}+2J(u_{0})u_{0}=0.
\end{equation}
where $M_{0}:=||u_0||_{L^2(\hhh)}^{2}$.

If $1+M_{0}\lambda_{0}\leq 0$, then multiply both sides of \eqref{m1} by $u_0$ and integrate, one gets $u_{0}\equiv0$.
If $1+M_{0}\lambda_{0} >0$, one can consider the equation for $\tilde{u}_{0}:=(\frac{1}{1+M_{0}\lambda_{0}})u_0$ and then argue as the proof of Lemma \ref{noninter} to show $\tilde{u}_{0}\equiv 0$.

If case 2 holds, by again using Lagrangian multiplier, one can derive an equation for $u_{0}$ and show that $u_{0}\equiv 0$ as in the proof of Lemma \ref{noninter}.

Case 3 cannot hold. Indeed, if case 3 holds, then $u_{0}$ is the optimizer for the sharp Sobolev $\|u\|_{L^{6}(\hhh)}\leq C^{*}\| \nabla_{\hhh} u\|_{L^2 (\hhh)}$, then writing down its Euler-Lagrange equation, we get
\begin{equation}
{\Delta u_{0}+ u_{0}^{5}=0.}
\end{equation}
Again, by Lemma \ref{nonextence}, $u_{0}=0$. Thus, case 3 cannot hold. 
\end{proof}

 \subsection{Main Proposition and proof for Theorem \ref{thmmain}}
Theorem \ref{thmmain} mainly relies on the following proposition
\begin{prop}\label{lemmamain}
Assume $\inf_{f\in \Omega}J(f)<\frac{1}{2}$, i.e. there exists a minimizing sequence {$u_{k} \in \Omega$}, such that 
$\limk J(u_{k})=\frac{1}{2}-\signot\equiv \inf_{f\in \Omega}J(f)$, and $\signot>0$. Then there exists a function $v\in \bar{\Omega}/\{0\}$, such that $J(v)=\frac{1}{2}-\signot$.
\end{prop}
Let us assume the above Proposition temporarily and  prove Theorem \ref{thmmain}.
\begin{proof}
Step 1: $\inf_{f\in\Omega} J(f)\leq \frac{1}{2}$.
Indeed, by studying the spectrum $-\Delta_{\hhh}$, we know the smallest eigenvalue of $-\Delta_{\hhh}=1$, i.e.
\begin{description}
\item[(a)] $$\forall f \neq 0, \, \frac{\|\nabla_{\hhh} f\|_{\hlt}}{\|f\|_{\hlt}}\geq 1.$$
\item[(b)] $$\exists f_{k}, \|f_{k}\|_{\hlt}=1, \limk \frac{\|\nabla_{\hhh} f_{k}\|_{\hlt}}{\|f_{k}\|_{\hlt}}=1.$$
\end{description}
Note by picking $f_{k}$ as in (b), it is not hard to see $\limk J(\frac{f_{k}}{k})=\frac{1}{2}$ and $\frac{f_{k}}{k}\in \Omega$ for large $k$.

Step 2: $J(f)\geq 0$. This follows directly from Lemma \ref{energytrapping}.

Step 3: Assume $\inf_{f\in \Omega}J(f)<\frac{1}{2}$, we apply Proposition \ref{lemmamain} to get $v$ such that $J(v)=\inf_{f\in \Omega}J(f)$, however, this contradicts Lemma \ref{noninter} and Lemma \ref{nobound}.
\end{proof}
All that is left is to prove is Proposition \ref{lemmamain}.
We give two reminders before we start.
\begin{enumerate}
\item 
Since we are in hyperbolic space, $\dot{H}^{1}(\hhh)$ is exactly $H^{1}(\hhh)$.
\item {A lot of limits are taken in this note, most arguments work only up to replacement by a subsequence, one may consider all the following arguments in which we take limits to only be valid up to replacement by a subsequence.}
 \end{enumerate}

\subsection{Outline for the proof of Proposition \ref{lemmamain} }
Throughout this subsection ,we  assume 
\begin{assumption}\label{assume}
$\{u_{k}\}$ is a minimizing sequence in $\Omega$, such that 
\begin{equation}\label{eqbase}
J(u_{k})\rightarrow \frac{1}{2}-\signot \equiv \inf_{f\in \Omega}J(f), \quad \signot>0.
\end{equation}
\end{assumption}
The  strategy for the proof is quite standard in variational arguments, the basic idea is that any minimizer cannot be split into two decoupled parts.

We first show $u_{k}$ cannot converge to zero. To be precise, we have
\begin{lem}\label{lemnontrivial}
Under Assumption \ref{assume}, there is $m_{0}>0$ such that $\liminfk \|u_{k}\|_{\hlt}\geq m_{0}>0, \\
\liminfk \|u_{k}\|_{\hnorm}\geq m_{0}>0$.
\end{lem}
See Subsection \ref{sectionnontrivial} for the proof.

Then we show no mass can escape to infinity, we have
\begin{lem}\label{lemnonescaping}
Under Assumption \ref{assume}, let $\chi$ be some fixed compactly supported bump function such that  $$\chi(x)=\begin{cases}
1, |x|\leq \frac{1}{2}\\
0, |x|\geq 1
\end{cases},$$ then for any $\epsilon>0$, there exists $R$, such that 
\begin{equation}\label{eqnonescaping}
\limsupk \left\|\left(1-\chi\left(\frac{x}{R}\right)\right)u_{k}\right\|_{\hnorm}<\epsilon.
\end{equation}
\end{lem}
See Subsection \ref{sectionnonescaping} for the proof.

Then we show no mass concentrates near the origin, we have
\begin{lem}\label{lemnonconcentration}
Under Assumption \ref{assume}, let $\chi$ be some fixed compactly supported bump function, $$\chi(x)=\begin{cases}
1, |x|\leq \frac{1}{2} \\
0, |x|\geq 1
\end{cases},$$ then for any $\epsilon>0$, there exists $R$, such that 
\begin{equation}\label{nomassislocal}
\limk \|\chi(Rx)u_{k}\|_{\hnorm}<\epsilon.
\end{equation}
\end{lem}
Finally, we use profile decomposition on Euclidean space to extract some non-trivial element minimizing our functional $J$, this is done in Subsection \ref{sectionfinalstory}.

\subsection{Proof of Lemma \ref{lemnontrivial}}\label{sectionnontrivial}
We  point out that once the first inequality is shown,
the second inequality follows from  Lemma \ref{lemmass}.
Now let us show the first inequality. We will prove by contradiction.
Without loss of generality (up to replacement by a sub-sequence), we assume all the $\liminf$, $\limsup$ are in fact $\lim$.
If the first inequality does not hold, one has $\limk \|u_{k}\|_{\hlt}=0$. Since $J(u_{k})$ is bounded, this implies that $\ek \xrightarrow{k} 0$. Thus, by energy trapping,  Lemma \ref{energytrapping}, we have $\|\nabla_{\hhh} u_{k}\|_{\hlt}\rightarrow 0$, which of course is equivalent to $\|u_{k}\|_{\hnorm}\rightarrow 0$. This implies that $\liminf_{k\rightarrow}J(u_{k})\geq \frac{1}{2}$ by Lemma \ref{lemmass},  a contradiction to $\limk J(u_{k})=\frac{1}{2}-\signot\equiv \inf_{f\in \Omega}J(f)$. 

\subsection{Proof of Lemma \ref{lemnonescaping}}\label{sectionnonescaping}
We need the following lemma
\begin{lem}\label{axlemma}
Under Assumption \ref{assume}, there exists a $R_{0}\gg 1, \epsilon_{0}\ll 1$, such that 
$$\liminfk \left\|\chi\left(\frac{x}{R_{0}}\right)u_{k} \right\|_{\hnorm}\geq  2\epsilon_{0}.$$
\end{lem}
\begin{proof}[Proof of Lemma \ref{axlemma}]
{We will show this by contradiction}. Assume Lemma \ref{axlemma} does not hold, then up to replacement by a subsequence, one may assume 
$$\left\|\chi\left(\frac{x}{k}\right)u_{k}\right\|_{\hnorm}\leq \frac{1}{k}.$$ 
By Sobolev embedding, we derive  that 
$$\limk \left\|\chi\left(\frac{x}{k}\right)u_{k}\right\|_{\hls}=0. $$ 
By Hardy inequality, Lemma \ref{hardy}, simply using pointwise control, we derive that 
$$\limk \left\|\left(1-\chi\left(\frac{x}{k}\right)\right)u_{k}\right\|_{\hls}=0.$$
Then since $\liminfk \|u_{k}\|_{\hlt}>0$ (Lemma \ref{lemnontrivial}), we have that (recall the smallest eigenvalue of $-\Delta_{\hhh}$ is $1$)
\begin{equation}
\limk J(u_{k})=\limk \frac{\frac{1}{2}\|\nabla_{\hhh} u_{k}\|_{\hlt}^{2}}{\|u_{k}\|_{\hlt}^{2}}\geq \frac{1}{2},
\end{equation}
a contradiction to \eqref{eqbase}.
\end{proof}
Now we will prove Lemma \ref{lemnonescaping} by contradiction.
Assume \eqref{lemnonescaping} does not hold, then there exists a sequence of $\{R_{k}\}_{k}$ and $\epsilon_{1}>0$, such that (up to replacement by a subsequence)
\begin{equation}
\limk \left\|\left(1-\chi\left(\frac{x}{R_{k}}\right)\right)u_{k}\right\|_{\hnorm}>2\epsilon_{1}, \quad \limk R_{k}=\infty.
\end{equation}
We can apply Lemma \ref{axlemma} to obtain $R_{0}\gg 1, \epsilon_{0}\ll 1$, such that 
$$\liminfk \left\|\chi\left(\frac{x}{R_{0}}\right)u_{k} \right\|_{\hnorm}\geq  2\epsilon_{0}.$$
Then by Lemma \ref{decouplemass}, we can find a decomposition such that
\begin{itemize}
\item $ u_{k}=f_{k,1}+f_{k_2}+r_{k}$.
\item $u_{k}=f_{k,1}$ for $|x|\leq 2R_{0}$, $u_{k}=f_{k,2}$ for $|x|\geq \frac{R_{k}}{2}$.
\item $f_{k,1}, f_{k,2}$ do not have common support for large k.
\item  $\limsupk \|r_{k}\|\leq \epsilon_{2}$. Here $\epsilon_{2}\ll \min\{\epsilon_{0},\epsilon_{1}\}$, and is a fixed small constant which will be chosen later.
\end{itemize}
Note it is not hard to see 
\begin{equation}\label{lowernonescaping}
\liminfk \|f_{k,1}\|_{\hnorm}\geq \epsilon_0, \quad\liminfk \|f_{k,2}\|_{\hnorm}\geq \epsilon_{1}.
\end{equation}
Also, the support condition between $f_{k,1}, f_{k,2}$ implies for large $k$ that we have
\begin{equation}\label{decopuleenergyphysical}
|\hyperenergy(u_{k})-\hyperenergy(f_{k,1})-\hyperenergy(f_{k,2})|\lesssim \epsilon_{2}^{2},
\end{equation}
\begin{equation}\label{decopulekineticphysical}
|\|\nabla_{\hhh} u_{k}\|_{\hlt}^{2}-\|\nabla_{\hhh} f_{k,1}\|_{\hlt}^{2}-\|\nabla_{\hhh} f_{k,2}\|_{\hlt}^{2}|\lesssim \epsilon_{2}^{2}.
\end{equation}\\
We also remark that since $R_{k}$ goes to $\infty$, $f_{k,2}$ is essentially located far away from the origin, so by Hardy inequality \eqref{hardy}, we have  $\limk \|f_{k,2}\|_{\hls}=0.$ Thus, (up to replacement by a subsequence),
\begin{equation}\label{equivfk2}
\limk \|\nabla_{\hhh} f_{k,2}\|_{\hlt}^{2}= \limk \hyperenergy(f_{k,2}).
\end{equation}
Thus, when $\epsilon_{2}$ is small enough,  (in particular, small enough compared to $\epsilon_{0},\epsilon_{1}$), we have:
\begin{itemize}
\item By \eqref{decopulekineticphysical}, there is some $\epsilon_{3}$ which only relies on $\epsilon_{1},\epsilon_{0}$, such that
$$\max(\|\nabla_{\hhh} f_{1,k}\|_{\hlt},\|\nabla_{\hhh} f_{2,k}\|_{\hlt}) <\|\nabla_{\hhh} u_{k}\|_{\hlt}-\epsilon_{3}.$$
\item By \eqref{decopuleenergyphysical}, and the fact that $$\limk E_{\hhh}(f_{k,2})=\limk\|\nabla_{\hhh}f_{k,2}\|_{\hlt}^{2}>0,$$ (see \eqref{lowernonescaping}, \eqref{equivfk2}) there is $\epsilon_{4}$ such that for large k, we have
$E_{\hhh}(f_{k,1})<E_{\hhh}(u_{k})-\epsilon_{4}$,
\item Thus, energy trapping (Lemma \ref{energytrapping}) holds for $f_{k,1}$ for large $k$, so there exists $c_{0}=c_{0}(\epsilon_{3},\epsilon_{4})$, such that for large k, $E_{\hhh}(f_{k,1})\geq c_{0}\|\nabla_{\hhh} f_{k,1}\|_{L^2(\hhh)}^{2}$.
\item The above in turn (using \eqref{decopuleenergyphysical} ) shows for large k, $E_{\hhh}(f_{k,2})<E_{\hhh}(u_{k})$.
\end{itemize}
In particular, the above facts imply that for  $k$ large, $f_{k,1}, f_{k,2}\in \Omega$.
\begin{rem}
The point is, essentially all our argument is quantitative, our choice of $\epsilon_{3}, \epsilon_{4}$ only depends on $\epsilon_{0},\epsilon_{1}$ .
Of course, $\epsilon_{3},\epsilon_{4}$ depends on the fact that $\epsilon_{2}$ is small, however, once $\epsilon_{2}$ is smaller than some threshold determined by $\epsilon_{0},\epsilon_{1}$, our choice of $\epsilon_{3}$, $\epsilon_{4}$ { does not depend on the exact value of $\epsilon_{2}$.}\\
\end{rem}
We claim that we already have a contradiction.
We discuss three different cases:

\textbf{Case 1:} $\limk \|f_{k,1}\|_{\hlt}=0$.\\
Then $J(u_{k})=\frac{E_{\hhh}(f_{k},1)+E_{\hhh}(f_{k,2})+O(\epsilon_{2}^{2})}{\|f_{k,2}\|_{\hlt}^{2}}+o_{k}(1)$, which implies
\begin{equation}
\liminfk \frac{\hyperenergy(f_{k,2})}{\|f_{k,2}\|_{\hnorm}^{2}}\leq\limk J(u_{k})-\frac{\epsilon_{0}^{2}}{C}+C\epsilon_{2}^{2},
\end{equation} 
where $C$ is some universal large number.
Therefore, when $\epsilon_{2}$ is chosen small enough, 
\begin{equation}
\liminfk J(f_{k,2})<\limk J(u_{k})=\inf_{f\in \Omega} J(f).
\end{equation}
A contradiction.

\textbf{Case 2:} $\limk \|f_{k,2}\|_{L^{2}(\hhh)}$=0.
A similar argument to Case 1 produces a contradiction. 

\textbf{Case 3:} If both Case 1 and Case 2 do not occur, then we may assume $\limk \|f_{k,1}\|_{\hlt}=a_{0}>0$, and $\limk \|f_{k,2}\|_{\hlt}=b_{0}>0$,
We claim (again, up to replacement by a subsequence we assume all the limits exist)
\begin{equation}\label{tempo2}
\limk \frac{\hyperenergy(f_{k,1})}{\|f_{k,1}\|}=\gamma<\frac{1}{2}-\signot\equiv \limk J(u_{k}),
\end{equation}
{Which will give us a contradiction.}

Let us prove \eqref{tempo2} now. Let $\gamma:=\limk \frac{\hyperenergy(f_{k,1})}{\|f_{k,1}\|}$.
Since $\|f_{k,2}\|_{\hls}\rightarrow 0$, thus $$\liminfk \frac{\hyperenergy(f_{k,2})}{\|f_{k,2}\|_{\hlt}^{2}}\geq \frac{1}{2},$$
which gives us
\begin{equation}
\frac{1}{2}-\sigma_{0}=\limk J(u_{k})\geq \frac{a_{0}^{2}\gamma+\frac{1}{2}b_{0}^{2}}{a_{0}^{2}+b_{0}^{2}}+O(\epsilon_{2}^{2}).
\end{equation}
Thus, when $\epsilon_{2}$ is small enough, $\gamma<\frac{1}{2}-\signot$, which is \eqref{tempo2}.

\subsection{Proof of Lemma \ref{lemnonconcentration}}
Lemma \ref{lemnonconcentration} is  an analogue of Lemma \ref{lemnonescaping}, {and will be proved using Lemma \ref{decopulemassversion2},} the analogue of Lemma \ref{decouplemass}.

Before we begin the  proof, let us remark that for any $R\geq 1$, uniformly we have
\begin{equation}\label{notsobad}
\|\chi(Rx)h(x)\|_{\hnorm}\lesssim \|h\|_{\hnorm}, \forall h\in H_{rad}^{1}(\hhh).
\end{equation}
Now let us turn to the proof of the Lemma.
\begin{proof}[Proof of Lemma \ref{lemnonconcentration}]
We prove by contradiction.
If  Lemma \ref{lemnonconcentration} does not hold, then we would be able to find $R_{k}\xrightarrow{k\rightarrow \infty} \infty$ and $\tilde{\epsilon}_{0}>0$ such that
\begin{equation}
\limk \|\chi(R_{k}x) u_{k}\|_{\hnorm}>\tilde{\epsilon}_{0}.
\end{equation}
{We then claim that there is some} $R_{0}>0$, $\epsilon_{1}>0$ such that
\begin{equation}\label{tempo3}
\limsupk \|\nabla_{\hhh}[(1-\chi(R_{0}x))u_{k}])\|_{\hlt}\geq \epsilon_{1}.
\end{equation}
Indeed, if \eqref{tempo3} does not hold for any $R_{0},\epsilon_{1}>0$, one can find a sequence $
\widetilde{R}_{k}\rightarrow \infty$ such that $\limk\| (1-\chi(\widetilde{R}_{k}x))u_{k}\|_{\hnorm}=0$. Then $\limk \|u_{k}\|_{\hlt}=\limk \|\chi(\widetilde{R}_{k}x) u_{k}\|_{\hlt}$. This immediately contradicts Lemma \ref{lemnontrivial} since by H\"older's inequality and  the boundedness of $\{u_{k}\}$ in $\hnorm$, one easily derives that $\limk \|\chi(\widetilde{R}_{k}x) u_{k}\|_{\hlt}=0$.
Thus \eqref{tempo3} holds.

Now we apply Lemma \ref{decopulemassversion2} and find a decomposition such that
\begin{itemize}
\item $u_{k}=f_{k,1}+f_{k,2}+r_{k}$
\item $u_{k}=f_{k,1}$, for $|x|\geq \frac{1}{2R_{0}}$
\item $u_{k}=f_{k,2}$, for $|x|\leq \frac{2}{R_{k}}$
\item $f_{k,2}$ is supported in {$|x|\leq \frac{2}{R'_{k}}$, $\limk R'_{k}=\infty$}.\\ 
\item  $\|r_{k}\|_{\hnorm}\leq \epsilon_{2}$ for large $k$, {where $\epsilon_{2}$ is some small constant to be chosen later.}
\end{itemize}
Note {we still have energy and $L^2$ gradient decoupling, so} \eqref{decopuleenergyphysical}, \eqref{decopulekineticphysical} still hold.\\

Though the final value of  $\epsilon_{2}$ will be chosen at the very end of this subsection, we  first require $\epsilon_{2}<\frac{m_{0}}{2}$. This immediately implies $\limsupk \|f_{k,1}\|_{\hlt}\geq \frac{m_{0}}{2}$, since  by H\"older's inequality and boundedness of $\{f_{k,2}\}$ in $\hnorm$,  we have $$\limk \|f_{k,2}\|_{\hlt}=0.$$
Thus, up to replacement by a sub sequence,  we obtain  $\limk \|\nabla_{\hhh} f_{k,1}\|_{\hlt}\gtrsim m_{0}$. 

On the other hand, by \eqref{tempo3}, up to picking a sub sequence, we have $$\limk \|\nabla_{\hhh} [\chi(R_{k}x)f_{k,2}]\|_{\hlt}>0,$$ {and so by} \eqref{notsobad}, $\limk\|\nabla_{\hhh} f_{k,2}\|_{\hlt}>0$. Thus, we may find some $\epsilon_{1}$ such that
$\min (\|\nabla_{\hhh} f_{k,1}\|_{\hlt}, \|\nabla_{\hhh} f_{k,2}\|_{\hlt})>\epsilon_{1}$. Now choose $\epsilon_{2}$ at least smaller than $\frac{\epsilon_{1}}{2}$. { By applying \eqref{decopulekineticphysical} we know there is some} $\epsilon_{3}>0$, such that for large k, $$\max(\|\nabla_{\hhh} f_{k,1}\|_{\hlt}, \|\nabla_{\hhh} f_{k,2}\|_{\hlt})< \|\nabla_{\hhh} u_{k}\|_{\hlt}-\epsilon_{3}.$$

Next we show for large $k$, there is an $\epsilon_{4}$, such that $$\max(E_{\hhh}(f_{k, 1}),E_{\hhh}(f_{k,2}))\leq E_{\hhh}(u_{k})-\epsilon_{4}.$$  When $\epsilon_{2}$ is small enough, i.e. the error $r_{k}$ is small enough, then up to picking a subsequence one may assume,  for large $k, E_{\hhh}(f_{k,1})>\frac{1}{3}\hyperenergy(u_{k})$ or $E_{\hhh}(f_{k,2})>\frac{1}{3}\hyperenergy(u_{k})$. Without loss of generality, we assume  $E_{\hhh}(f_{k,1})>\frac{1}{3}\hyperenergy(u_{k})$. This implies $E(f_{k,2})<E(u_{k})-\tilde{\epsilon}_{4}$ for some $\tilde{\epsilon}_{4}>0$. Thus $f_{k,2}\in \Omega$ {and we can apply our energy trapping result}, Lemma \ref{energytrapping} on $f_{k,2}$. This  gives us a lower bound $E_{\hhh}(f_{k,2})\gtrsim \|\nabla_{\hhh}{f_{k,2}}\|_{\hnorm}^{2}\geq \epsilon^{2}_{1}$, and so, when $\epsilon_{2}$ is chosen even smaller, we can bound $E_{\hhh}(f_{k,1})$ strictly above by $E_{\hhh}(u_{k})$, by \eqref{decopuleenergyphysical}. 
Thus we confirm that we can find an $\epsilon_{4}$
such that $\max(E_{\hhh}(f,k_{1}),E_{\hhh}(f_{k,2}))\leq E_{\hhh}(u_{k})-\epsilon_{4}$.
In particular for large $k$, $f_{k,1}, f_{k,2}$ are in $\Omega$.

Now we claim $\limk J (f_{k,1}) < \frac{1}{2}-\signot$, which is a contradiction.
Indeed, $$J(u_{k})=\frac{E_{\hhh}(u_{k})}{\|u_{k}\|_{\hlt}}=\frac{E_{\hhh}(f_{k,1})+E_{\hhh}(f_{k,2})+O(\epsilon_{2}^{2})}{\|f_{k,1}\|_{\hlt}+O(\epsilon_{2}^{2})},$$ 
note that $E_{\hhh}(f_{k,2})\gtrsim \|\nabla_{\hhh} f_{k,2}\|_{\hlt}^{2}\geq \epsilon_{1}^{2}$, then choosing $\epsilon_{2}$ small enough, and for large $k$, $J(f_{k,1})<\limk J(u_{k})$, {making $\{f_{k,1}\}_{k}$ a better minimizing sequence which is a contradiction}.

\end{proof}

\subsection{Profile decomposition and the end of proof for Proposition \ref{lemmamain}}\label{sectionfinalstory}
In this Subsection, under Assumption \ref{assume},  we use profile decomposition on the Euclidean space to construct the minimizer $v$.

We first make a  simple observation, $\{u_{k}\}_{k}$ can be viewed as a bounded sequence  in $H^{1}(\RRR^{3})$, then up to picking subsequence, we may assume $\{u_{k}\}$ admits a profile decomposition with $\{v_{j};\{\lambda_{j,k}\}_{k}\}_{j}$ in the sense of Lemma \ref{concentrationcompactness}.
We will show one profile actually gives the minimizer.

Without loss of generality we may assume there are only three types of profiles:
\begin{itemize}
\item type (a): Profile $v_{j}$ such that $\limk\lambda_{j,k}=\infty$.
\item type (b): Profile $v_{j}$ such that $\lambda_{j,k}\equiv 1$ for all $k$.
\item type (c): Profile $v_{j}$ such that $\limk \lambda_{j,k}=0$.
\end{itemize} 
Note there is at most one profile of type (b), since  $$j\neq j'\Longrightarrow \limk \frac{\lambda_{j,k}}{\lambda_{j',k}}+\frac{\lambda_{j',k}}{\lambda_{j,k}}
=+\infty.$$
We first show
\begin{lem}\label{aaa}
Type (a) profiles do not exist.
\end{lem}
\begin{proof}
{We will focus on the function spaces} $\hdot$ and $\rls$. The proof is  by contradiction.
Suppose $v_{1}$ is a profile of type (a), if $v_{1}\neq 0$, then there is $R_{}0>0,\epsilon_{0}>0$, such that
\begin{equation}\label{nontriciall6}
\|v_{1} 1_{\frac{1}{R_{0}}<|x|<R_{0}}\|_{L^{6}(\RRR^{3})}>\epsilon_{0}\\
\end{equation}
Now let $\chi$ be some smooth bump function  
$$\chi(x)=\begin{cases}
1, \quad 1/R_{0}\leq |x|\leq R_{0},\\
0, \quad |x|\leq 1/(2R_{0}) \text{ or } |x|\geq 2R_{0},
\end{cases}$$
Then by the property of profile decomposition,
\begin{equation}\label{alescape}
\|\chi(x/\lambda_{1,k})u_{k})\|_{\rls}= \left\|\frac{1}{\lambda_{1,k}}v_{1} \left(\frac{x}{\lambda_{1,k}}\right)\chi\left(\frac{x}{\lambda_{1,k}}\right)\right\|_{\rls}+o_{k}(1)\equiv \|\chi v_{1}\|_{\rls}+o_{k}(1).
\end{equation}
 Since $u_{k}$ is bounded in $\hnorm$ and using pointwise control from the Hardy inequality (Lemma \eqref{lemhardyinequality}), we obtain $$\lim_{R\rightarrow\infty}\sup_{k}\|u_{k}1_{|x|\geq R}\|_{\rls}=0,$$
which is clearly a contradiction to \eqref{alescape}, since $\lim_{k\rightarrow \infty} \lambda_{1,k}=\infty$.
\end{proof}
We then show
\begin{lem}
Type (c) profile does not exist .
\end{lem}
\begin{proof}
One argues line by line as in Lemma \ref{aaa}, and ends up with
\begin{equation}
\|\chi(x/\lambda_{1,k})u_{k})\|_{\rls}=\left\|\frac{1}{\lambda_{1,k}}v_{1}\left(\frac{x}{\lambda_{1,k}}\right)\chi\left(\frac{x}{\lambda_{1,k}}\right)\right\|_{\rls}+o_{k}(1)\equiv \|\chi v_{1}\|_{\rls}+o_{k}(1).
\end{equation}
This indicates there exists $R_{k}\rightarrow \infty$, such that for large k, $\|\chi(R_{k}x)u_{k}\|_{\rls}>\epsilon>0$ for some $\epsilon$, which is equivalent to
 $\|\chi(R_{k}x)u_{k}\|_{\hls}>\epsilon_{1}>0$ for some $\epsilon_{1}$, a contradiction to Lemma \ref{lemnonconcentration}. 
\end{proof}

Only two possibility are left now.
\begin{itemize}
\item \textbf{Case 1}: There are no non-trivial profiles, i.e.
\begin{equation}
\limk \|u_{k}\|_{\rls}\rightarrow 0.
\end{equation}
\item \textbf{Case 2}: There is one nontrivial profile $v,$ i.e. $u$ admits a profile decomposition:
\begin{equation}
u_{k}=v+r_{k}
\end{equation}
where $r_{k}\rightharpoonup 0$ in $\hdot$ and $\limk \|r_{k}\|_{\rls}=0$.
\end{itemize}
We first exclude Case 1.
By Lemma \ref{lemnonescaping}, for any $\epsilon>0$, there is a compactly supported cut-off function $\chi $ such that
\begin{equation}\label{eqtempo314}
\limk \|u_{k}-\chi u_{k}\|_{\hnorm}\leq \epsilon.
\end{equation}
Meanwhile, since $\chi$ is compact supported, we have
\begin{equation}\label{eqtempo318}
\|\chi u_{k}\|_{\hlt}\lesssim \|\chi u_{k}\|_{\rlt} \lesssim \|\chi u_{k}\|_{\rls}\lesssim \| u_{k}\|_{\rls}.
\end{equation}
{However we have assumed that $\limk \|u_{k}\|_{\rls}=0$
and so \eqref{eqtempo314} and \eqref{eqtempo318} imply }
$
\limk \|u_{k}\|_{\hlt}=0
$, which clearly contradicts $$\limk \|u_{k}\|_{L^{2}(\hhh)}>m_{0}>0, $$  see Lemma \ref{lemnontrivial}.

Now the only possibility is \textbf{Case 2}.
We again fix some bump function $$\chi=
\begin{cases} 
1, |x| \leq 1,\\
0, |x| \geq 2.
\end{cases}
$$
{We first observe that although profile decomposition only ensures $v\in \hdot \cap \rls$, we can actually show} $v\in \hnorm$.\\
Indeed, $u_{k}$ weakly converges to $v$ in $\hdot$ (which of course implies weak convergence in $\rls$).
And $u_{k}$ is uniformly bounded in $\hnorm$. Thus $v$ must be in $\hnorm$.
In particular, $r_{k}$ is also uniformly bounded in $\hnorm$.

Furthermore, we claim:
\begin{equation}\label{vanishl6fortheremainder}
\limk \|r_{k}\|_{\hls}=0.
\end{equation}
Indeed,  for any $R$, we have $r_{k}=r_{k}\chi(\frac{x}{R})+(1-\chi(\frac{x}{R}))r_{k}$.
By $\limk \|r_{k}\|_{\rls}=0$, we get for any $R$ fixed
\begin{equation}\label{eqvanish1}
\limk \left\|\chi\left(\frac{x}{R}\right)r_{k}\right\|_{\hls}=0.
\end{equation}
{Since $r_{k}$ is uniformly bounded in $\hnorm$, we can apply the Hardy inequality (Lemma \ref{lemhardyinequality}) to see }
\begin{equation}\label{eqvanish2}
\lim_{R\rightarrow \infty} \sup_{k} \left\|\left(1-\chi\left(\frac{x}{R}\right)\right)u_{k}\right\|_{\hls}=0,
\end{equation}
and \eqref{vanishl6fortheremainder} clearly follows from \eqref{eqvanish1} and \eqref{eqvanish2}.\\
Now we show the following orthogonality property
\begin{equation}\label{critical orthogonality}
\limk (v, r_{k})_{\hnorm}=0.
\end{equation}
Since $v\in H^{1}(\hhh)$ and $\{r_{k}\}_{k}$ is a bounded sequence in $\hnorm$, we can always use a sequence $v_{n}\in C_{c}^{\infty}$ to approximate $v$. Therefore we need only show that for any $w\in C_{c}^{\infty}$  we have
\begin{equation}
\limk (w,r_{k})_{\hnorm}=0,
\end{equation}
This directly follows from \eqref{vanishl6fortheremainder} since $$(w,r_{k})_{H^{1}}=\int_{\hhh}((-1+\Delta_{\hhh})w)r_{k}.$$
A similar argument shows that for any R, we have
\begin{equation}
\limk \left(\chi\left(\frac{x}{R}\right) v, \chi\left(\frac{x}{R}\right) r_{k}\right)_{\hnorm}=0.
\end{equation}
Next, note that  $v\in \bar{\Omega}/{0}$.

We want to show $J(v)\leq \frac{1}{2}-\delta_{0}\equiv \limk J(u_{k})$.
From the view point of Lemma \ref{lemnonescaping}, we need only show for all $R>0$
\begin{equation}\label{endofstory}
\limk J\left(\chi\left(\frac{x}{R}\right)u_{k}\right)\geq \limk J\left(\chi\left(\frac{x}{R}\right)v\right).
\end{equation}
(One may need to replace by a subsequence to ensure the limit exists.)

 Since $\|r_{k}\|_{\hls}\xrightarrow{k} 0$, we have $$\limk E(\chi(\frac{x}{R})r_{k})\geq 0,  \limk \|\chi(\frac{x}{R}) r_{k}\|_{\hlt}=0,$$ {and since $u_k=v+r_k$} \eqref{endofstory} clearly follows. This concludes the proof.
 
We remark the above argument actually shows $\limk \|r_{k}\|_{\hnorm}=0$ and implies
\begin{equation}\label{finalstory}
\limk \|u_{k}-v\|_{\hnorm}=0.
\end{equation}

\section{Acknowledgment}
We thank Gigliola Staffilani for consistent support, helpful discussion and careful reading of the material. Part of this work was done in MIT summer SPUR program 2014, we thank Pavel Etingof and David Jerison for helpful discussion. Part of this work was done while Chenjie Fan was in residence at the Mathematical Sciences Research Institute in Berkeley, California, during Fall 2015 semester, under the support of NSF Grant No. 0932078000.
\appendix
\section{A collection of technical lemmas}
\begin{lem}\label{lemhardyinequality}
For any $u\in H^{1}_{rad}(H^{3})$, 
\begin{equation}\label{hardy}
|u(r)|\lesssim e^{-\frac{3}{4}r}, \forall r\geq 10.\\
\end{equation}
\end{lem}
\begin{proof}
We believe the result is classical and the power we get here is not optimal. We include a quick proof for the reader's convenience.
Without loss of generality we assume $\|u\|_{H^{1}(\hhh)}=1$ and we  show $|u(r_{0})|\lesssim  e^{-\frac{3}{4}r_{0}}, \forall r_{0}\geq 10$.
Indeed, since $u$ is radial and $\|u\|_{H^{1}}(\hhh)=1$, one immediately gets
\begin{equation}
\int_{0}^{\infty}(|\partial_{r}u(r)|^{2}+|u(r)|^{2})\sinh^{2} rdr\lesssim 1
\end{equation}
The idea is to view $u$ as a function on $\RRR$, we will rely on the usual one dimensional Sobolev embedding $H^{1}(\RRR)\hookrightarrow L^{\infty}(\RRR)$.\\
Indeed, let $\chi(x)$ be some smooth bump function on $\RRR$, such that
\begin{equation}
\chi(x) =
\begin{cases}
1, |x|\leq 1,\\
0, |x|\geq 2.
\end{cases} 
\end{equation}
Then one immediately obtains
\begin{equation}\label{tempo1}
\int_{\RRR} |\chi(r-r_{0})u(r)|^{2}+|\partial_{r} (\chi(r-r_{0})u(r))|^{2}\sinh^{2}rdr\lesssim 1.
\end{equation}
Furthermore by \eqref{tempo1}, one gets $$\|\chi(r-r_{0})u(r)\|_{L^{\infty}(\RRR)}\lesssim \|\chi(r-r_{0})u(r)\|_{H^{1}(\RRR)}\lesssim (\sinh (r_{0}-2))^{-1}\lesssim e^{-\frac{3}{4}r_{0}}.$$\\
This completes the proof.
\end{proof}
\begin{lem}\label{vanishingofmass}
For any bounded sequence $\{f_{k}\}_{k}$ in $H^{1}_{rad}({\hhh})$, we have
\begin{equation}
\lim_{R\rightarrow \infty}\sup_{k} \|1_{|x| \geq R}f_{k}(x)\|_{\hls}=0.
\end{equation}
\end{lem}
Lemma \ref{vanishingofmass} directly follows from Lemma \ref{lemhardyinequality}.

\begin{lem}\label{lemmass}
$\forall \sigma>0$, there exists $\epsilon>0$, such that for all $f\in \hnorm, f\neq 0$
\begin{equation}
\|f\|_{\hnorm}\leq \epsilon\Longrightarrow J(f)\geq \frac{1}{2}-\sigma
\end{equation}
\end{lem}
\begin{proof}
Essentially, when $f$ falls into the small data regime (i.e. $\|\nabla_{\hhh} f\|_{H^{1}}\leq \epsilon$), the $L^{6}$ part in the $J(f)$ can be neglected, and $J(f)\approx \frac{1}{2}\|\nabla_{\hhh} f\|^{2}_{\hlt}/\|f\|^{2}_{\hlt}$, while the latter is not smaller than $\frac{1}{2}$ due to the spectral property of $-\Delta_{H^{3}}$.\\
\end{proof}
\begin{lem}\label{decouplemass}
Given $R>0$, and $R_{k}\rightarrow \infty$, $\epsilon>0$, for any bounded sequence $\{f_{k}\}_{k}$ in $H^{1}_{rad}(H^{3})$ one can always (up to replacement by a sub-sequence) find a decomposition
\begin{equation}
f_{k}=f_{k,1}+f_{k,2}+r_{k} 
\end{equation}
such that for $k$ large enough,
\begin{itemize}
\item  $f_{k}=f_{k,1}$ for $|x|\leq R$
\item $f_{k}=f_{k,2}$ for $|x|\geq R_{k}$
\item $f_{k,1}$ and $f_{k,2}$ do not have common support.
\item $f_{k,2}$ vanishes for $|x|\leq \widetilde{R}_{k}$, $\limk \widetilde{R}_{k}=\infty$
\item $\|r_{k}\|_{H^{1}(H^{3})}\leq \epsilon$.
\end{itemize}
\end{lem}
The proof is elementary and essentially the pigeonhole principle, see Section \ref{bbbbbb}.
\begin{lem}\label{decopulemassversion2}
Given $R>0$, and $R_{k}\rightarrow \infty$, $\epsilon>0$, for any bounded sequence $\{f_{k}\}_{k}$ in $H^{1}_{rad}(H^{3})$ one can always (up to replacement by a sub-sequence) find a decomposition
\begin{equation}
f_{k}=f_{k,1}+f_{k,2}+r_{k} 
\end{equation}
such that for $k$ large enough,
\begin{itemize}
\item  $f_{k}=f_{k,1}$ for $|x|\geq R^{-1}$
\item $f_{k}=f_{k,2}$ for $|x|\leq R_{k}^{-1}$
\item $f_{k,1}$ and $f_{k,2}$ do not have common support.
\item $f_{k,2}$ vanishes for $x\geq \widetilde{R}_{k}^{-1}$, $\widetilde{R}_{k}\rightarrow \infty$.
\item $\|r_{k}\|_{H^{1}(H^{3})}\leq \epsilon$.
\end{itemize}
\end{lem}
The proof is essentially the pigeonhole principle, and the argument is similar to that of \eqref{decouplemass}.

\section{Proof of Lemma \ref{nonextence} }\label{aaaaa}
We show how to upgrade Lemma \ref{nonextencepre} to Lemma \ref{nonextence}. 

\begin{proof}
With Lemma \ref{nonextencepre}, we need only to show that  if some $u(x)\in H^{1}(\hhh)$ , $u(x)\geq 0$ and $u(x_{0})=0$ for some $x_{0}\in \hhh$, then $u\equiv 0$. 

First, by classical elliptic regularity theory  (\cite{trudinger1968remarks}), one  has that $u$ is smooth.
Now, since $u$ is radial, $u(x)=u(r)$, we have for some $\lambda\in \RRR$
\begin{equation}\label{2222}
u_{rr}+2\coth r u_{r}=-u^{5}+\lambda u.
\end{equation}
Note if at some $r_{0}$, $u(r_{0})=0$, then we must also have $u_{r}(r_{0})=0$ since at $r_{0}$ the nonnegative function $u$ obtains its minimal value. Since  \eqref{2222} is a second order ordinary differential equation, this implies $u\equiv 0$.

The case $r_{0}=0$ warrants special discussion, as at first glance the equation \eqref{2222} appears singular. However, since $u$ is smooth, $u$ is infinitely differentiable at $r_{0}=0$ and one easily deduces that $u(0)=0, u_{r}(0)=0$ and $u_{rr}(0)=0$. Therefore \eqref{2222} is not singular at $0$ and one still gets $u\equiv 0$.
\end{proof}

\section{Proof of Lemma \ref{decouplemass}} \label{bbbbbb}
\begin{proof}
This section is devoted to the proof of Lemma \ref{decouplemass}. Let us assume $R\sim 2^{j_{0}}, R_{k}\sim 2^{j_{k}}$, where $j_{0}, j_{k}$ are integers. Since $R$ is fixed and $\limk R_{k}=\infty$, we have $\limk j_{k}-j_{0}=\infty.$ Thus, by pigeonhole principle, (note $\|f_{k}\|_{H^{1}(H^{3})}$ is bounded), when $k$ is large enough, we will be able to find $j_{0}\ll l_{k}\ll j_{k}$, such that
\begin{equation}
\|f_{k}\|_{H^{1}(H^{3})}\leq \epsilon/1000.
\end{equation}
Then, simply do a (smooth) cut off at at $|x|\sim 2^{l_{k}}$, and one will obtain the desired decomposition. To be precise, let 
\begin{equation}
\chi_{k,1}=
\begin{cases}
1, |x|\leq 2^{l_{k}},\\
0, |x|\geq 2^{l_{k}+1/2}.
\end{cases}
\end{equation}
\begin{equation}
\chi_{ k,2 }=
\begin{cases}
1, |x|\geq 2^{l_{k}+1},\\
0, |x|\leq 2^{l_{k}+1/2}.
\end{cases}
\end{equation}
and let
\begin{equation}
f_{k,1}=f\chi_{k,1}, \quad f_{k,2}=f\chi_{k,2}, \quad r_{k}=f-f_{k,1}-f_{k,2}.
\end{equation}
\end{proof}
\bibliographystyle{alpha}
\bibliography{BG}
\end{document}